\documentclass[11pt,reqno,a4paper]{amsart}
\usepackage[latin1]{inputenc}
\usepackage[T1]{fontenc}
\usepackage{amssymb}
\usepackage{bm}
\usepackage[alpine]{ifsym}
\usepackage{xpicture}
\usepackage{calculator}
\usepackage{graphicx}
\usepackage{xcolor}

\usepackage[a4paper,top=3.3cm,bottom=3cm,left=3cm,right=3cm,bindingoffset=5mm]{geometry}

\def\sideremark#1{\ifvmode\leavevmode\fi\vadjust{\vbox to0pt{\vss 
      \hbox to 0pt{\hskip\hsize\hskip1em           
 \vbox{\hsize2cm\tiny\raggedright\pretolerance10000
 \noindent #1\hfill}\hss}\vbox to8pt{\vfil}\vss}}} %
\usepackage{color}

\usepackage[colorlinks=true]{hyperref}
\newtheorem{theorem}{Theorem}[section]
\newtheorem{lemma}[theorem]{Lemma}
\newtheorem{corollary}[theorem]{Corollary}
\newtheorem{proposition}[theorem]{Proposition}

\theoremstyle{definition}
\newtheorem{example}[theorem]{Example}
\newtheorem{remark}[theorem]{Remark}
\newtheorem{definition}[theorem]{Definition}

\numberwithin{equation}{section}

\begin{document}
\title[Divisorial valuations of Hirzebruch surfaces]{Non-positive and negative at infinity divisorial valuations of Hirzebruch surfaces}

\author[C. Galindo]{Carlos Galindo}

\address{Instituto Universitario de Matem\'aticas y Aplicaciones de Castell\'on, y Departamento de Matem\'aticas, Jaume I University, Spain.}\email{galindo@mat.uji.es}  \email{cavila@uji.es}

\author[F. Monserrat]{Francisco Monserrat}
\address{Instituto Universitario de
Matem\'atica Pura y Aplicada, Universidad Polit\'ecnica de Valencia,
Camino de Vera s/n, 46022 Valencia (Spain).}
\email{framonde@mat.upv.es}

\author[C.-J. Moreno-\'Avila]{Carlos-Jes\'us Moreno-\'Avila}


%
%

%
\thanks{Partially supported by the Spanish Government Ministerio de Econom\'ia, Industria y Competitividad (MINECO), grants  MTM2015-65764-C3-2-P, MTM2016-81735-REDT and BES-2016-076314, as well as by Universitat Jaume I, grant P1-1B2015-02.}

\begin{abstract}
We consider rational surfaces $Z$ defined by divisorial valuations $\nu$ of Hirzebruch surfaces. We introduce concepts of non-positivity and negativity at infinity for these  valuations and prove that these concepts admit nice local and global equivalent conditions. In particular we prove that, when $\nu$ is non-positive at infinity, the extremal rays of the cone of curves of $Z$ can be explicitly given.
\end{abstract}

\maketitle

\section{Introduction}

Valuations were introduced by Dedekind and Weber for studying Riemann surfaces but it was Kürschák who gave the first axiomatic definition. In the middle of the past century, Zariski and Abhyankar \cite{Abh1,Abh2,Zar2,Zar3,ZarSam} used the theory of valuations as a main tool to treat resolution of singularities of algebraic varieties, and much more recently, after the proof by Hironaka of resolution in characteristic zero, they are still considered suitable for the positive characteristic case \cite{Tei}.

Valuations are essentially local objects which could be used  to prove local uniformization. However, in the last years, they have been used to study global properties. The best known situation corresponds with valuations of the fraction field $K(\mathcal{O}_{\mathbb{P}^2,p})$ of the local ring $\mathcal{O}_{\mathbb{P}^2,p}$ centered at $\mathcal{O}_{\mathbb{P}^2,p}$, $\mathbb{P}^2$ being the projective plane over an algebraically closed field $k$ and $p$ a closed point in $\mathbb{P}^2$. These valuations were classified by Spivakovsky \cite{Spiv} (see also \cite{FavJon1,Grekiy}). This classification has five types and works for valuations of the fraction field of any two-dimensional regular local ring $R$ centered at $R$.

Divisorial and irrational valuations are two of these types and, for them and suitable divisors, it can be defined Seshadri-like constants \cite{CutEinLaz}, that is, objects which basically contain  the same information for valuations as Seshadri constants for points. Recall that Seshadri constants were used by Demailly \cite{Dem} for studying the Fujita's conjecture. Even in the most simple case, where the local ring is $\mathcal{O}_{\mathbb{P}^2,p}$ and the divisor is a line, Seshadri-like constants are very difficult to compute; they allow us to establish the concept of minimal valuation and propose a conjecture which implies the Nagata conjecture and is implied by that of Greuel-Lossen-Shustin (see \cite{GalMonMoy} and also \cite{DumHarKurRoeSze}, where the mentioned Seshadri-like constant is denoted by $\hat{\mu}(\nu)$).

Spivakovsky's classification also contains the so-called exceptional curve valuations. Exceptional curve valuations are given by a pair whose first projection, $\nu_1$, is a divisorial valuation. When the local ring is $\mathcal{O}_{\mathbb{P}^2,p}$, they correspond with flags of the form $\{X=X_r\supset E_r\supset \{q\}\}$, where $q$ is a closed point and $X_r$ is the surface obtained after a finite simple sequence of point blowing-ups starting at $p$, $E_r$ being the last obtained exceptional divisor which defines $\nu_1$. Newton-Okounkov bodies \cite{KavKho,Oko3} are the analogue to Seshadri constants for these exceptional curve valuations and, again, are very difficult to explicitly compute \cite{CilFarKurLozRoeShr,GalMonMoyNic2}.

Recently, in \cite{GalMon}, it was considered a class $\mathcal{N}$ of divisorial valuations $\nu$ (of $K(\mathcal{O}_{\mathbb{P}^2,p})$ centered at $\mathcal{O}_{\mathbb{P}^2,p}$) which have a similar behaviour as that of curves with only one place at infinity \cite{AbhMoh}. They were named non-positive at infinity  because satisfy $\nu(f)\leq 0$ for every $f\in k[x,y]\setminus \{0\}$,  $\{x,y\}$ being affine  coordinates in the chart of points which are not in the line at infinity (which contains $p$). When $\nu(f)<0$, they are called negative at infinity. These valuations are centered at infinity because the point $p$ is in the line at infinity \cite{FavJon2}. Recently, this last class of valuations has been studied and used in different contexts \cite{CamPilReg, FavJon2, FavJon3, Jon, Mond}. The set of divisorial valuations $\nu$ centered at $\mathcal{O}_{\mathbb{P}^2,p}$ and that of finite simple sequences of point blowing-ups starting with the blowing-up at $p$ are bijective, and each valuation $\nu$ determines a rational projective surface $X$. In \cite{GalMon} we proved that the fact that $\nu$ belongs to $\mathcal{N}$ is equivalent to that of the cone of curves $NE(X)$ is regular, and we gave a simple characterization of this fact. Even more, we are able to compute the Seshadri-like constant with respect to a line divisor for  valuations in $\mathcal{N}$ \cite{GalMonMoy}, and to explicitly give the Newton-Okounkov bodies of flags where the valuation given by the divisor $E_r$ belongs to $\mathcal{N}$ \cite{GalMonMoyNic2}.

Since the projective plane and the Hirzebruch surfaces provide the classical minimal models for rational surfaces, we consider divisorial valuations of the fraction field $K(\mathcal{O}_{\mathbb{F}_\delta,p})$ centered at $\mathcal{O}_{\mathbb{F}_\delta,p}$ (called here divisorial valuations of $\mathbb{F}_\delta$), where $\mathbb{F}_\delta$ is any Hirzebruch surface and $p\in\mathbb{F}_\delta$  is a closed point. Our objective is to find suitable affine charts on a Hirzebruch surface such that, as in the case of the class $\mathcal{N}$, valuations $\nu$ which are non positive (or negative) on non-zero regular functions on these charts give rise to surfaces (obtained by the sequence of blowing-ups given by $\nu$ which starts at the Hirzebruch surface) with nice geometrical global properties. Notice that, as algebraic objects, valuations of $\mathbb{F}_\delta$ do not differ from valuations centered at regular closed points of other birationally equivalent surfaces; however we desire to relate the mentioned valuations with global geometric aspects of Hirzebruch surfaces. We will show that, in this case, there exist two natural charts "at infinity". On the one hand that given by points which are neither in the fiber $F_1$ that contains $p$ nor in the special section $M_0$ of $\mathbb{F}_\delta$, and, on the other hand, by points which are neither in $F_1$ nor in a particular uniquely defined section $M_1 \neq M_0$. In this paper, we will divide the divisorial valuations of $\mathbb{F}_\delta$ in two classes, special and non-special, according to the chart at infinity to be used for introducing the concepts of non-positive and negative at infinity divisorial valuation of $\mathbb{F}_\delta$. That is, over each point $p \in \mathbb{F}_\delta$, we will consider special or non-special divisorial valuations which will determine the chart to be used to define non positivity or negativity at infinity.
We will give several characterizations of those concepts, including one which is very easy to check  from the dual graph of the valuation $\nu$ (that involves only topological information) and the images by $\nu$ of (the germs at $p$ of) the fiber and sections before introduced (see Item (c) in theorems \ref{Th1_caso_especial} and \ref{Th1_caso_general}), and, in the case of negative at infinity valuations, the Iitaka dimension of certain divisor (Item (b) in theorems \ref{Th2_caso_especial} and \ref{Th2_caso_general}).

Each divisorial valuation $\nu$ of $\mathbb{F}_\delta$ defines in a unique way a rational surface $Z$ obtained from the simple sequence of point blowing-ups given by $\nu$.  A remarkable property of non-positive at infinity valuations is that they determine (in fact, are equivalent to) the surfaces $Z$ as above  such that its cone of curves  $NE(Z)$ is finite  polyhedral and generated either by the classes of the strict transforms of the fiber $F_1$, the special section $M_0$ and the exceptional divisors (special valuations) or by the mentioned generators plus the class of the section $M_1$ (non-special valuations). Since the Hirzebruch surface $\mathbb{F}_1$ can be obtained by blowing-up a point in $\mathbb{P}^2$, our results recover those in \cite{GalMon} concerning the characterization of non-positive and negative at infinity divisorial valuations of $\mathbb{P}^2$, and provide a very simple characterization of the rational surfaces (obtained from a classical minimal model by a finite simple sequence of point blowing-ups) whose cone of curves has the above mentioned generators.

Some complementary results on the effective monoid of surfaces given by blowing-up some very concrete configurations of infinitely near points over Hirzebruch surfaces can be found in \cite{mus}.

For surfaces defined by non-positive at infinity valuations of $\mathbb{P}^2$, we are able to  decide whether their Cox rings are finitely generated \cite{GalMon}, and, as mentioned, for these same valuations, we know how to compute their Seshadri-like constants and  to explicitly obtain their corresponding Newton-Okounkov bodies. In a forthcoming paper, we will prove that similar properties can be deduced when considering non-positive at infinity valuations of $\mathbb{F}_\delta$.

Section \ref{preli} of the paper contains the ingredients we need to develop it. Special and non-special divisorial valuations of  $\mathbb{F}_\delta$ are introduced in Definition \ref{SG}. Section \ref{speci} studies the special ones and characterizes its non-positivity (respectively, negativity) at infinity in Theorem \ref{Th1_caso_especial} (respectively, Theorem \ref{Th2_caso_especial}). The non-special divisorial valuations are considered in Section \ref{gener}, being  theorems \ref{Th1_caso_general} and \ref{Th2_caso_general} the main results.

\section{Preliminaries}
\label{preli}

Given a surface $Z_0$, a (finite) simple sequence of blowing-ups starting at $Z_0$ is a sequence
\begin{equation}\label{Eq_sequencepointblowingups}
\pi: Z=Z_n\xrightarrow{\pi_n} Z_{n-1}\rightarrow \ldots \rightarrow Z_1 \xrightarrow{\pi_1} Z_0,
\end{equation}
of blowing-ups $\pi_i: Z_i\to Z_{i-1},$ $1\leq i\leq n$, centered at closed points $p_i\in Z_{i-1}$, such that $p_1=p\in Z_0$  and each $p_i,$ $2\leq i\leq n,$ belongs to the exceptional divisor created by $\pi_{i-1}.$

In this paper, we will study some global properties concerning rational surfaces obtained from simple sequences $\pi$ as above where $Z_0$ is a Hirzebruch surface. We start by recalling some basic facts about these surfaces (see \cite{Har,Beau,Reid,Moe} for additional information).

\subsection{Hirzebruch surfaces}
Let $k$ be an algebraically closed field and $\mathbb{P}^1=\mathbb{P}_k^1$ the projective line over $k$. Let $\delta$ be a non-negative integer; the $\delta$th \emph{Hirzebruch surface} is the projective ruled surface over $\mathbb{P}^1$, $\mathbb{F}_\delta : =\mathbb{P}(\mathcal{O}_{\mathbb{P}^1}\oplus\mathcal{O}_{\mathbb{P}^1}(-\delta))$,  together with the projection morphism $\mathrm{pr}:\mathbb{F}_\delta\to \mathbb{P}^1$. It is well-known that the Picard group, Pic$(\mathbb{F}_\delta),$ of $\mathbb{F}_\delta$ is isomorphic to $\mathbb{Z}\oplus\mathbb{Z}$ and admits as generators the class of a general fiber $F$ and that of a section $M_0$ of $\mathrm{pr}$ (which, if $\delta>0$, is the unique section with non-positive self-intersection and is called \emph{special section}). We recall that, with respect to this basis, the  intersection matrix is
\begin{equation*}
\left(\begin{array}{cc}
0 & 1\\
1 & -\delta
\end{array}\right).
\end{equation*}
The group  $\text{Pic}(\mathbb{F}_\delta)$ is also generated by $[F]$ and $[M]$, where $M$ denotes a linearly equi\-valent to $\delta F+M_0$ general section of $\mathrm{pr}$; the symbol $[\cdot]$ will denote the class in the Picard group, throughout the paper. Here, the intersection matrix is
\begin{equation*}
\left(\begin{array}{cc}
0 & 1\\
1 & \delta
\end{array}\right).
\end{equation*}

From a coordinates point of view, Hirzebruch surfaces $\mathbb{F}_\delta$ can be obtained as the quotient of the product of punctured affine planes over $k$, $$(\mathbb{A}^2\setminus \{(0,0)\})\times(\mathbb{A}^2\setminus \{(0,0)\}),$$ by an action of the product, $k^*\times k^*$, of multiplicative groups of the field $k$  (see \cite[\S 2.2]{Reid}).  For each $(\lambda,\mu)\in k^*\times k^*$, the action goes as follows:
\begin{equation}\label{Action_group_Fdelta}
\begin{array}{cccc}
(\lambda,1): & (X_0,X_1;Y_0,Y_1)& \to & (\lambda X_0,\lambda X_1; Y_0,\lambda^{-\delta}Y_1) \\
(1,\mu): & (X_0,X_1;Y_0,Y_1)& \to & (X_0,X_1; \mu Y_0,\mu Y_1 ).
\end{array}
\end{equation}
Note that the action preserves the ratio $(X_0:X_1)$ and that the morphism $\mathrm{pr}:\mathbb{F}_\delta \to \mathbb{P}^1$ is the projection onto the first factor. According \cite[\S 1.2]{Moe}, the homogeneous coordinate ring of $\mathbb{F}_\delta$ is the graded ring on $\mathbb{Z}\times\mathbb{Z}_{\geq 0}$,
$S_\delta:=k[X_0,X_1,Y_0,Y_1]$, where the degree of each variable is:
$$
\deg X_0=(1,0), \, \deg X_1=(1,0), \, \deg Y_0=(0,1) \text{ and}\, \deg Y_1=(-\delta,1),
$$
and, therefore, the set of the homogeneous elements in $S_\delta$ of degree $(a,b)\in \mathbb{Z}\times\mathbb{Z}_{\geq 0}$ is
\begin{align}\label{Eq_CoordinateringFdelta}
S_\delta(a,b):= H^0(\mathbb{F}_\delta,\mathcal{O}_{\mathbb{F}_\delta}(a,b))=\bigoplus_{
\alpha_0+\alpha_1=\delta \beta_1+a, \, \beta_0+\beta_1=b} k X_0^{\alpha_0}X_1^{\alpha_1}Y_0^{\beta_0}Y_1^{\beta_1}.
\end{align}
As a consequence, an integral curve $C$ on $\mathbb{F}_\delta$ is defined by an irreducible and reduced homogeneous polynomial $H\in S_\delta(a,b)$. In this case, we will say that $C$ has degree $(a,b)$. Notice that any irreducible curve $C,$ $C\neq F,M_0$, of degree $(a,b),$ satisfies that $C$ is linearly equivalent to $ aF+bM,$ where $a\geq 0$ and $b>0$. This is a consequence of \cite[V, Proposition 2.20]{Har} and the fact that $M$ is linearly equivalent to $\delta F+M_0$. So, in the sequel, we will use $[F]$ and $[M]$ as generators of $\text{Pic}(\mathbb{F}_\delta)$ because, in this case, degree and coordinates of the classes of irreducible curves will be the same.

Keeping the above notation, one can consider the four affine open sets $\{U_{ij}\}_{0\leq i,j\leq 1}$ of the surface $\mathbb{F}_\delta,$ where $U_{ij}:=\mathbb{F}_\delta\setminus \textbf{V}(X_iY_j)$. For instance, within $U_{00}$ we get affine coordinates $(1:X_1/X_0;1,(X_0^\delta Y_1)/Y_0)\cong (u,v)$, where $(u,v)\in \mathbb{A}_k^2$. Similarly we can give affine coordinates for the remaining affine open sets, and also deduce formulae for the changes of coordinates. In this paper, we will only use that, if  $(u,v)$ are  coordinates in $U_{ij}$ and $(u',v')$ are coordinates in $U_{rl}$, where $i,j,r,l\in\{0,1\}$, then
\begin{equation*}
u'=\dfrac{1}{u} \;\;\; \mbox{and} \;\; v'=\dfrac{1}{u^\delta v},
\end{equation*}
when $i=j=0$ and $r=l=1$; moreover
\begin{equation*}
u'=\dfrac{1}{u} \;\;\; \mbox{and} \;\; v'=\dfrac{v}{u^\delta},
\end{equation*}
whenever $i=0$ and $j=r=l=1$.

Considering the classical definition of minimal surfaces, it is well-known  that a minimal rational surface is isomorphic to $\mathbb{P}^2$ or to $\mathbb{F}_\delta$ for $\delta\neq 1$ and that $\mathbb{F}_1$ is isomorphic to $\mathbb{P}^2$ with a point blown-up (see \cite[Theorem V.10 and Proposition IV.1]{Beau}).

When $\delta=0$, $\mathbb{F}_0=\mathbb{P}(\mathcal{O}_{\mathbb{P}^1}\otimes\mathcal{O}_{\mathbb{P}^1})$ and we have no special section because $M$ and $M_0$ are linearly equivalent. Looking  at the ring $S_0$, we deduce that a point $p$ on $\mathbb{F}_0$ determines and it is determined by a unique curve $F$ of degree $(1,0)$ and a unique curve $M_0$ of degree $(0,1)$.

If $\delta\geq 1$, the equation $Y_1=0$ defines the special section $M_0$ and its points will be called {\it special points} of $\mathbb{F}_\delta$. Non-special points will be called {\it general}. Notice that a change of coordinates on $\mathbb{F}_\delta$ must take general (respectively, special) points to general (respectively, special) points. We conclude with the following result, stated in \cite{Moe}, which can be deduced from \eqref{Eq_CoordinateringFdelta}.
\begin{proposition}\label{Lemma_Moe}
Assume that $\delta \geq 1$. Then there exists a $\delta$-dimensional family of irreducible curves of degree $(0,1)$ passing through a general point of $\mathbb{F}_{\delta}$. Also, a curve of degree $(1,0)$ and an irreducible curve of degree $(0,1)$ meet at a general point of $\mathbb{F}_{\delta}$.
\end{proposition}

\subsection{Divisorial valuations.}

A {\it valuation} of a field $K$ is a surjective map $$\nu:K\setminus \{0\}\to G,$$ where $G$ is a totally ordered commutative group (the value group of $\nu$), such that, for $f,g\in K\setminus\{0\}$, satisfies:
$$
\nu(f+g)\geq \min\{\nu(f),\nu(g)\} \text{ and also } \nu(fg)=\nu(f)+\nu(g).
$$
The ring $R_\nu=\{f\in K\setminus \{0\} \ | \ \nu(f)\geq 0\}\cup \{0\}$ is called the {\it valuation ring} of $\nu$. It is a local ring whose maximal ideal is $\mathfrak{m}_\nu=\{f\in K\setminus\{0\} \ | \ \nu(f)>0\}\cup \{0\}$. When $K$ is the fraction field of a local regular ring $(R,\mathfrak{m})$ and $R\cap\mathfrak{m}_\nu=\mathfrak{m}$, one says that $\nu$ is {\it centered} at $R$. When $\dim R=2$, valuations centered at $R$ ($R/\mathfrak{m}$ algebraically closed) are in one-to-one correspondence with (non-necessarily finite) simple sequences of point blowing-ups starting at $\text{Spec} R$. Divisorial valuations are those corresponding with finite simple sequences \cite{ZarSam,Spiv}.

In this paper $p$ will be a closed point in $\mathbb{F}_\delta$. Denote by $K$ the fraction field of the local ring $R=\mathcal{O}_{\mathbb{F}_\delta,p}$. Then, to give a sequence as \eqref{Eq_sequencepointblowingups} where $Z_0=\mathbb{F}_\delta$ is equivalent to give a divisorial valuation $\nu$ of $K$ centered at $R$. The valuation $\nu$ is defined by the last exceptional divisor $E$ in the sequence $\pi:Z\to Z_0$ and, frequently and by simplicity, we will say that $\nu$ is a valuation of $\mathbb{F}_\delta$. The map $\pi_1$ is the blowing-up of $Z_0$ at $p=p_1$ and $\pi_{i+1},$ $1\leq i\leq n-1$, the blowing-up of $Z_i$ at the unique point $p_{i+1}$ of the exceptional divisor defined by $\pi_i,$ $E_i,$ such that $\nu$ is centered at the local ring $\mathcal{O}_{Z_i,p_{i+1}}$. Write $\mathcal{C}_\nu:=\{p_i\}_{i=1}^n$ the sequence (or configuration) of infinitely near points above defined; $p_i$ is said to be \emph{proximate} to $p_j$, denoted by $p_i\to p_j$, whenever $i>j$ and $p_i$ belongs either to $E_j$ or to the strict transform of $E_j$ on $Z_{i-1}$. A point $p_{i}$ is \emph{satellite} whenever there exists $j<i-1$ satisfying $p_i\to p_j$; otherwise, it is called \emph{free}. As in the case of germs of plane curves  \cite{Cam}, plane divisorial valuations admit sets of invariants that help to study them. For a valuation $\nu$ as above, we will use its sequence of maximal contact values $\{\overline{\beta}_j\}_{j=0}^{g+1}$ \cite[(1.5.3)]{DelGalNun} and its sequence of Pusieux exponents $\{\beta_j'\}_{j=0}^{g+1}$ \cite[(1.5.2)]{DelGalNun}. Notice that both sequences can be obtained one from the other \cite[Theorem 1.11]{DelGalNun}. The continued fraction expansions of the values $\{\beta_j'\}_{j=0}^{g+1}$ determine (and are  determined by) the dual graph of $\nu$. The dual graph of a valuation $\nu$ as above is a labelled tree, where each vertex represents an exceptional divisor appearing in the sequence of blowing-ups \eqref{Eq_sequencepointblowingups} and two vertices are joined whenever their corresponding divisor meet. Each vertex is labelled with the number of blowing-ups needed to create the corresponding divisor. The set  $\{\overline{\beta}_j\}_{j=0}^{g}$ generates the {\it semigroup of values} of $\nu$, $S(\nu)=\nu(R\setminus\{0\})$, \cite[Remark 6.1]{Spiv}, and the sequence of maximal contact values has an extra value $\overline{\beta}_{g+1}$ which coincides with the inverse of the volume of $\nu$, $[\text{vol}(\nu)]^{-1}.$ Indeed, by definition,
$$
\text{vol}(\nu)=\lim_{\alpha\to\infty}\dfrac{\dim_k(R/ \mathcal{P}_\alpha)}{\alpha^2/2},
$$
where $\mathcal{P}_\alpha=\{f\in R \ | \ \nu(f)\geq\alpha\}\cup \{0\}$; taking into account that the above dimensions depend only on local data, the equality $\overline{\beta}_{g+1}= [\text{vol}(\nu)]^{-1}$
follows as in \cite[Remark 2.3]{GalMonMoyNic2}.

Along the paper we denote by $\varphi_i$, $1 \leq i \leq n$, an analytically irreducible germ of curve at $p$ whose strict transform on $Z_i$ is transversal to $E_i$ at a non-singular point of the exceptional locus. Also, for any curve $C$ on $\mathbb{F}_{\delta}$, $\varphi_{C}$ denotes its germ at $p$ and $(\varphi_i,\varphi_{C})_p$ equals $0$ (respectively, the intersection multiplicity at $p$ of the germs $\varphi_i$ and $\varphi_{C}$) if $C$ does not pass through $p$ (respectively, otherwise). Finally, $\text{mult}_{p_j}(\varphi_i)$ (respectively, $\text{mult}_{p_j}(\varphi_C)$), $1\leq i,j \leq n,$ means multiplicity of the strict transform of $\varphi_i$ (respectively, $\varphi_C$) at $p_j$. We will use frequently, without any mention, the so-called Noether's formula for valuations, that we recall here for convenience of the reader and whose proof can be found in \cite[Theorem 8.1.6]{Cas}:

\begin{lemma}
Let $\nu$ be a divisorial valuation of $K$ centered at $R$, with associated configuration $\mathcal{C}_\nu:=\{p_i\}_{i=1}^n$, and let $C$ be a curve on $\mathbb{F}_{\delta}$. Then
$$\nu(\varphi_C)=(\varphi_n,\varphi_C)_p=\sum_{j=1}^n {\rm mult}_{p_j}(\varphi_n)\cdot {\rm mult}_{p_j}(\varphi_C).$$
\end{lemma}

\subsection{Non-positive and negative at infinity valuations.}
\label{nonposneg}

A divisorial valuation of the fraction field of $\mathcal{O}_{\mathbb{P}^2,p},$ centered at $\mathcal{O}_{\mathbb{P}^2,p},$ $p\in\mathbb{P}^2$ being closed point, (or, simply, a divisorial valuation of $\mathbb{P}^2$), is called {\it non-positive at infinity} when $\nu(h)\leq 0$ for all $h\in\mathcal{O}_{\mathbb{P}^2}(\mathbb{P}^2\setminus L),$ $L$ being a line (the line at infinity) containing $p$. In case it satisfies $\nu(h)<0$ for every non-constant function $h\in\mathcal{O}_{\mathbb{P}^2}(\mathbb{P}^2\setminus L)$, $\nu$ is named {\it negative at infinity}. As we mentioned in the introduction, non-positive and negative at infinity divisorial valuations of $\mathbb{P}^2$ have nice global properties involving the surfaces they define.

Afterwards we will introduce the concepts of non-positivity and negativity at infinity  for valuations of the fraction field of $\mathcal{O}_{\mathbb{F}_\delta,p}$, centered at $\mathcal{O}_{\mathbb{F}_\delta,p}$, $p$ being a point in $\mathbb{F}_{\delta}$. As we will see, our definition will depend on the point $p$ and a chart of the Hirzebruch surface which does not contain $p$. The goal of the paper is to show that the rational surfaces given by these valuations are easy to characterize and also enjoy nice global properties. We focus on the cone of curves and positivity properties of divisors.


\section{The sign at infinity of special valuations}
\label{speci}

We start this section by partitioning the set of divisorial valuations of Hirzebruch surfaces in two subsets because our main results have to do with the concept of "sign at infinity" of valuations which depends on the considered subset. First we need to fix some notations.

Let $\mathbb{F}_\delta$ be a Hirzebruch surface and $p \in \mathbb{F}_\delta$ a point on it. Consider a divisorial valuation $\nu$ of the fraction field of $\mathcal{O}_{\mathbb{F}_\delta,p}$ centered at $\mathcal{O}_{\mathbb{F}_\delta,p}$ and their associated configuration ${\mathcal C}_{\nu}=\{p_i\}_{i=1}^n$ and composition of blowing-ups $\pi:Z\to Z_0=\mathbb{F}_\delta$ as in \eqref{Eq_sequencepointblowingups}. Set $E_i$, $1 \leq i \leq n$, the exceptional divisor produced after blowing-up $p_i$ and denote by $\tilde{E}_i$ (respectively, $E_i^*$) its strict (respectively, total) transform on the surface $Z$. Often $E_i$ also means the strict transform of the divisor $E_i$. $\tilde{C}$ and $C^*$ also means strict and total transforms on $Z$ of other divisors $C$. Occasionally, these transforms could be on some surface $Z_i$ appearing in \eqref{Eq_sequencepointblowingups}.

\begin{definition}
\label{SG}
 A divisorial valuation $\nu$ as before is called to be {\it special} (with respect to $\mathbb{F}_\delta$ and $p$) when one of the following conditions holds:
\begin{enumerate}
\item $\delta = 0$.
\item $\delta >0$ and $p$ is a special point.
\item $\delta >0$, $p$ is a general point and there is no integral curve in the complete linear system $|M|$ whose strict transform on $Z$ has negative self-intersection.


\end{enumerate}
The remaining valuations will be called {\it non-special}.
\end{definition}

\begin{remark}
{\rm Looking at the local equations of the linear system $|M|$ and taking into account  their evolution by blowing-ups, it is not difficult to show that the above condition (3) holds if and only if either $p_2$ belongs to strict transform of the fiber of pr passing through $p$ on $Z_1$, or there does not exist $j\geq \delta+1$ such that the points $p_i$, $1\leq i\leq j$, of $\mathcal{C}_\nu$ are free.

}
\end{remark}

Throughout this section $\nu$ will be a \emph{special} divisorial valuation of $\mathbb{F}_\delta$. In addition, $F_1$ will denote the fiber passing through $p$, and $M_0$ will denote either the special section (in case $\delta\geq 1$), or the section of degree $(0,1)$ passing through $p$ (otherwise).\\

Denote by $\text{Pic}(Z)$ the Picard group of $Z$ and $\cdot$ the intersection pair associated to $\text{Pic}(Z)$. By extension consider the linear space $\text{Pic}_\mathbb{Q}(Z)=\text{Pic}(Z)\otimes_\mathbb{Z}\mathbb{Q}$ and, by abuse of notation, $\cdot$ will denote the corresponding bilinear pairing. Recall that the convex cone of $\text{Pic}_\mathbb{Q}(Z)$ generated by the classes of effective divisors (respectively, nef divisors) is called the \emph{cone of curves} (respectively, \emph{nef cone}) and denoted by $NE(Z)$ (respectively, $P(Z)$). Notice that $P(Z)$ is the dual cone of $NE(Z)$. In the following we will denote by $\overline{NE}(Z)$ the closure of $NE(Z)$ for the usual topology.

 By \cite[Lemma 1.22]{KolMor}, the classes $[\tilde{F}_1]$ and $[\tilde{M}_0]$ span extremal rays of both cones $NE(Z)$ and $\overline{NE}(Z)$. For our purposes, it will useful to consider the strongly convex cone of $\text{Pic}_\mathbb{Q}(Z)$, $S_1(Z)$, generated by the set of classes $\{[\tilde{F}_1],[\tilde{M}_0]\}\cup\{[E_i]\}_{i=1}^n$, and also its dual cone $$S_1^\vee(Z):=\{[C]\in \text{Pic}_\mathbb{Q}(Z)\ | \ [C]\cdot [D]\geq 0 \text{ for all }[D]\in S_1(Z)\}.$$


Our next result provides generators for $S_1^\vee(Z)$.

\begin{proposition}\label{Prop_generatorsofdualcone_casespecial}
The dual cone $S_1^\vee(Z)$ is generated by $[F^*],[M^*]$ and the classes $\{[\Lambda_i]\}_{i=1}^n$ of the divisors
\begin{equation}\label{Div_coneofcurvesregular}
\Lambda_i := a_iF^* + b_iM^* - \sum_{j=1}^i\text{\emph{mult}}_{p_j}(\varphi_i)E_j^*,
\end{equation}
where $
a_i:=(\varphi_i,\varphi_{M_0})_p \text{ and} \; b_i:=(\varphi_i,\varphi_{F_1})_p.$

\end{proposition}

\begin{proof}

It is enough to prove that $\{[F^*],[M^*]\}\cup \{[\Lambda_i]\}_{i=1}^n$ is the dual basis of $\{[\tilde{F}_1],[\tilde{M}_0]\}\cup\{[E_i]\}_{i=1}^n$ with respect to the intersection product.

Let $p_{i_{F_1}}$  be the last point in the configuration $\mathcal{C}_\nu$ of the valuation $\nu$ giving rise to $Z$  through which the strict transform of $F_1$ passes. Also, if $p$ belongs to $M_0$, we define $i_{M_0}$ such that $p_{i_{M_0}}$ is the last point of $\mathcal{C}_\nu$ through which the strict transform of $M_0$ passes; otherwise we define $i_{M_0}:=0$. Taking into account that $\varphi_i$ is analytically irreducible, the proximity equalities  \cite[Theorem 3.5.3]{Cas} show that $\Lambda_i\cdot E_j=\delta_{ij}$, where $\delta_{ij}$ denotes the Kronecker's delta. Also, for each $i\in \{1,2,\ldots,n\}$, it holds
$$\Lambda_i\cdot \tilde{F}_1=b_i-\sum_{j=1}^{\min\{i,i_{F_1}\}} \text{mult}_{p_j}(\varphi_i)=0,$$
and
$$\Lambda_i\cdot \tilde{M}_0=a_i-\sum_{j=1}^{\min\{i,i_{M_0}\}} \text{mult}_{p_j}(\varphi_i)=0,$$
where the summations with upper index equal to $0$ are defined to be $0$. Finally notice that $F^*\cdot \tilde{F}_1=0$, $F^*\cdot \tilde{M}_0=1$, $M^*\cdot \tilde{F}_1=1$, $M^*\cdot \tilde{M}_0=0$ and $F^*\cdot E_i=M^*\cdot E_i=0$ for all $i=1,2,\ldots,n$.  This concludes the proof.
\end{proof}

Recall that we are considering a special divisorial valuation  $\nu$ of $\mathbb{F}_\delta$ and the surface $Z$ that $\nu$ defines. The divisors $\Lambda_i, 1 \leq i \leq n$, defined in \eqref{Div_coneofcurvesregular}, will be useful in this section because, as we are going to prove, they satisfy nice properties.

\begin{lemma}\label{Lemma_2_casoespecial}
Let $\nu$ be a special divisorial valuation of $\mathbb{F}_\delta$. Then, with the above notation, it holds that $\Lambda_1^2 \geq 0$, and the inequality $\Lambda_i^2\geq 0$ for some index $i\in\{2,3,\ldots,n\}$ implies:
\begin{itemize}
\item[(a)] $\Lambda_i^2>0$, whenever $p_i$ is a  satellite point of the configuration $\mathcal{C}_\nu$.
\item[(b)] $\Lambda_{i-1}^2\geq 0$ and in case $\Lambda_{i-1}^2=0,$ the point $p_i$ is satellite and the point $p_{i-1}$ is free.
\end{itemize}
\end{lemma}
\begin{proof}
The self-intersection of the divisor $\Lambda_1$ satisfies $\Lambda_1^2=1+\delta$ when $p_1$ is a special point and also when $\delta=0$. Otherwise, $\Lambda_1^2=\delta-1$.

For proving the remaining statements, we can assume, without loss of generality, that $i=n\geq 2$.

We are going to prove the result when $p_1$ is a special point. Otherwise, the proof is the same after setting $\delta=0$ or $a_n=0$.

We start with the proof of Statement $(a)$ for which we will use some properties of the set of maximal contact values of $\nu$, $\{\overline{\beta}_j\}_{j=0}^{g+1}$, (see \cite{Spiv} and \cite[Section 1.5]{DelGalNun}). We divide this proof in two cases.

{\it Case 1(a)}: $g>1$. Reasoning by contradiction and taking into account that the point $p_n$ is satellite, we get that
$$
0= \Lambda_n^2 = 2a_nb_n+\delta b_n^2 - e_{g-1}\overline{\beta}_{g} = e_{g-1}\left[\dfrac{2a_nb_n+\delta b_n^2}{e_{g-1}} -\overline{\beta}_g\right],
$$
where $e_{g-1}:=\gcd(\overline{\beta}_0,\overline{\beta}_1,\ldots,\overline{\beta}_{g-1})$. Since both $a_n$ and $b_n$ are either a multiple of $\overline{\beta}_0$ or $\overline{\beta}_1,$ the first addend in the brackets is a multiple of $e_{g-1}$, which gives a contradiction because $\gcd(e_{g-1},\overline{\beta}_g)=1$.

{\it Case 2(a)}: $g=1$. We distinguish three sub-cases: {\it The values $a_n$ and $b_n$ are divisible by $\overline{\beta}_0$.} Then $e_{g-1}=e_0=\overline{\beta}_0$ and the proof follows as above. {\it The value $a_n$ satisfies $a_n=\overline{\beta}_1$.} Then $\Lambda_n^2=\overline{\beta_0}(2\overline{\beta}_1 +\overline{\beta}_0\delta-\overline{\beta}_1)> 0.$ {\it Otherwise.} Then $\Lambda_n^2=\overline{\beta_1}(2\overline{\beta}_0 +\overline{\beta}_1\delta-\overline{\beta}_0)> 0$, which concludes the proof of Statement $(a).$

Now we prove Statement $(b)$. Again we can suppose that $i=n$. We also assume that the point $p_n$ is satellite because, otherwise, $\Lambda_{n-1}^2>0$ by Noether's formula. Denote by $\widehat{\nu}$ the divisorial valuation defined by the divisor $E_{n-1}.$ Let $\{\widehat{\overline{\beta}}_j\}_{j=0}^{\widehat{g}+1}$ be the sequence of maximal contact values of $\widehat{\nu}$, set $\widehat{e}_{g-1}:=\gcd(\widehat{\overline{\beta}}_0,\widehat{\overline{\beta}}_1,\ldots,\widehat{\overline{\beta}}_{g-1})$ and $e:=\widehat{e}_{\widehat{g}-1}/e_{\widehat{g}-1}.$ Consider two cases with two sub-cases.

{\it Case 1(b)}: $g=\hat{g}$. Assume first that $g>1$. From the following equality, which is proved in \cite[Lemma 2]{GalMon},
\begin{equation}\label{eq_lemma2}
|\widehat{\overline{\beta}}_g-e\overline{\beta}_g|=\dfrac{1}{e_{g-1}},
\end{equation}
one can deduce that
\begin{equation}\label{Consq_eq_Lemma2}
-\dfrac{ e_{g-1}\widehat{\overline{\beta}}_g}{e} \geq - \dfrac{1}{e} -  e_{g-1}\overline{\beta}_g.
\end{equation}
In this case both valuations $\nu$ and $\widehat{\nu}$ are defined by satellite points, therefore $a_{n-1}=ea_{n}, b_{n-1}=eb_n,$  $\overline{\beta}_{g+1}=e_{g-1}\overline{\beta}_g$ and $\widehat{\overline{\beta}}_{g+1}=\widehat{e}_{g-1}\overline{\beta}_g$. As a consequence
\begin{align*}
\Lambda_{n-1}^2=e^2\left[ 2a_nb_n + \delta b_n^2 - \dfrac{e_{g-1}\widehat{\overline{\beta}}_g}{e}\right] \geq e^2 \left[ 2a_nb_n +\delta b_n^2  - \dfrac{1}{e} -  e_{g-1}\overline{\beta}_g \right]\\ =e^2\left[\Lambda_n^2-\dfrac{1}{e} \right] >0,
\end{align*}
where the first inequality is deduced from the inequality \eqref{Consq_eq_Lemma2} and the last one holds since $\Lambda_n^2 > e_{g-1} > 1/e$.

To conclude the proof in this case, it remains to study what happens when $g=1$. We consider the same subcases as above: The values $a_n$ and $b_n$ are both divisible by $\overline{\beta}_0$, then the fact $\Lambda_{n-1}^2>0$ can be proved as before.
The value $a_n$ equals $\overline{\beta}_1$, then
$$
\Lambda^2_{n-1}= 2\widehat{\overline{\beta}}_0\widehat{\overline{\beta}}_1 +\delta\widehat{\overline{\beta}}_0^2 - \widehat{\overline{\beta}}_2 =\widehat{\overline{\beta}}_0(2\widehat{\overline{\beta}}_1 + \delta\widehat{\overline{\beta}}_0 -\widehat{\overline{\beta}}_1)=\widehat{\overline{\beta}}_0(\widehat{\overline{\beta}}_1 + \delta\widehat{\overline{\beta}}_0)>0.
$$
Otherwise, then $\Lambda^2_{n-1}=\widehat{\overline{\beta}}_1(\widehat{\overline{\beta}}_0 + \delta\widehat{\overline{\beta}}_1)>0.$

{\it Case 2(b):} $\widehat{g}=g-1$. When $g>1,$ it holds
$$
 \widehat{\overline{\beta}}_{\hat{g}+1}=\dfrac{\overline{\beta}_{g+1}+2}{4}
$$
and thus
$$
\Lambda_{n-1}^2=\dfrac{1}{4}\left(2a_nb_n + \delta b_n^2 -\overline{\beta}_{g+1}-2\right)=\dfrac{1}{4}\Lambda_n^2-\dfrac{1}{2}\geq 0,
$$
because $\Lambda^2_n\geq 2$.

Finally we must assume that $g=1$ and, as above, when the values $a_n$ and $b_n$ are divisible by $e_0=2$, $\Lambda_{n-1}^2\geq 0$. When $a_n=\overline{\beta}_1$, $\Lambda_{n-1}^2 = \widehat{\overline{\beta}}_1+\delta \geq 0$, and otherwise,
$$
\Lambda_{n-1}^2=2\widehat{\overline{\beta}}_1+\delta\widehat{\overline{\beta}}_1^2-\widehat{\overline{\beta}}_1\geq 0,
$$
which concludes the proof.
\end{proof}

Next we introduce the concepts of non-positivity and negativity at infinity for special divisorial valuations. Afterwards a similar concept will be given for non-special valuations. We consider a Hirzebruch surface $\mathbb{F}_\delta$,  a closed point $p$ in $\mathbb{F}_\delta$ and denote by $R$ the local ring $\mathcal{O}_{\mathbb{F}_\delta,p}$.

\begin{definition}
\label{defnonposspecial}
Let $\nu$ be a special divisorial valuation of the fraction field of $R$ centered at $R$. The valuation $\nu$ is called {\it non-positive} (respectively, {\it negative}) {\it at infinity} whenever $\nu(h)\leq 0$ (respectively, $\nu(h)< 0$) for all $h\in\mathcal{O}_{\mathbb{F}_\delta}(\mathbb{F}_\delta\setminus (F_1\cup M_0))$ (respectively, $h\in\mathcal{O}_{\mathbb{F}_\delta}(\mathbb{F}_\delta\setminus (F_1\cup M_0))$, $h \notin k$).
\end{definition}

We devote the remaining of this section to state two results, Theorems \ref{Th1_caso_especial} and \ref{Th2_caso_especial}, which give several equivalent conditions to the fact that a special divisorial valuation of a Hirzebruch surface is  non-positive or negative at infinity. We will use the divisor $\Lambda_n$ and the values $a_n$ and $b_n$, defined in Proposition \ref{Prop_generatorsofdualcone_casespecial}.


\begin{theorem}\label{Th1_caso_especial}
Let $\nu$ be a special divisorial valuation of the fraction field of $R$ centered at $R$.  Set $Z$ the surface that $\nu$ defines. Consider the divisor $\Lambda_n$ given in  \eqref{Div_coneofcurvesregular} and the inverse of the volume of $\nu$, $[\text{\emph{vol}}(\nu)]^{-1}$. Then, the following conditions are equivalent:
\begin{itemize}
\item[(a)] The valuation $\nu$ is non-positive at infinity.
\item[(b)] The divisor $\Lambda_n$ is nef.
\item[(c)] The inequality $2a_nb_n+b_n^2\delta\geq [\text{\emph{vol}}(\nu)]^{-1}$ holds.
\item[(d)] The cone of curves $NE(Z)$ is generated by the classes of the strict transforms on $Z$ of the fiber passing through $p$, the special section and the irreducible exceptional divisors associated with the map $\pi$ given by $\nu$.
\end{itemize}
\end{theorem}
\begin{proof}
Our first step is to prove the equivalence between items (a) and (b), and we start by proving that Item (b) implies Item (a). We assume firstly here that $\delta>0$ and $p=p_1$ is a special point. Without loss of generality, suppose that the special point $p$ has coordinates $(1:0;1,0).$  The point $p$ belongs to the fiber $F_1$ whose equation is $X_1=0$, and the special section $M_0$ is defined by the equation $Y_1=0$. Set $U_{00}$ the affine open set of $\mathbb{F}_\delta$ given by $X_0\neq 0$ and  $Y_0\neq 0$, whose associated affine coordinates are $\{u,v\}=\big\{\frac{X_1}{X_0},\frac{X_0^\delta Y_1}{Y_0}\big\}$. Consider also the affine open set of $\mathbb{F}_\delta$, $U_{11}$, defined by $X_1\neq 0$ and $Y_1\neq 0$, with coordinates $\{x,y\}=\big\{\frac{X_0}{X_1},\frac{Y_0}{X_1^\delta Y_1}\big\}$. It holds that $p\in U_{00}$ and $F_1$ and $M_0$ have local equations $u=0$ and $v=0$, respectively. Denote by $\mathcal{P}$ the set of non-constant functions in $\mathcal{O}_{\mathbb{F}_\delta}(U_{11})$ (up to multiplication by a nonzero element of $k$) such that neither $x$ nor $y$ divide them. In terms of the coordinates $\{u,v\},$ $f\in\mathcal{P}$ can be expressed as
\begin{equation}\label{Ecu_th1_casoespecial}
f(x,y)=f(1/u,1/u^\delta v)=\dfrac{h_f(u,v)}{u^{\deg_1(h_f) +\delta \deg_2(h_f)}v^{\deg_2(h_f)}},
\end{equation}
where $h_f(u,v)\in\mathcal{O}_{\mathbb{F}_\delta}(U_{00}).$ The bi-homogeneous polynomial $$X_0^{\deg_1(h_f)}Y_0^{\deg_2(h_f)}\cdot h_f\left(\frac{X_1}{X_0},\frac{X_0^\delta Y_1}{Y_0}\right)$$ defines a curve $C_f$ on the surface $\mathbb{F}_\delta$ of degree $(\deg_1(h_f),\deg_2(h_f))$ and, if $F'$ and $M'$ are the fiber and the section on $\mathbb{F}_\delta$ with equations $X_0=0$ and $Y_0=0,$ it holds that the map $f\to C_f$ defines a one-to-one correspondence between $\mathcal{P}$ and the set of the curves on $\mathbb{F}_\delta$ containing no curve in $\{F_1,F',M_0,M'\}$ as a component. Now, the condition $\Lambda_n$ nef and Equality \eqref{Ecu_th1_casoespecial} show that
\begin{align*}
0\leq & \,\Lambda_n\cdot C_f=\Lambda_n\cdot \left[\deg_1(h_f)F^*+\deg_2(h_f)M^*-\sum_{i=1}^n\mbox{mult}_{p_i}(h_f)E_i^*\right]\\
=&-[-(\deg_1(h_f)+\deg_2(h_f)\delta )\nu (u) - \deg_2(h_f)\nu (v) + \nu (h_f)]= -\nu (f).
\end{align*}
So, to finish the proof of Item (a) in this case ($p$ is a special point), it only remains to assume that either $x$ or $y$ or both are factors of $f$. Then the proof follows from the existence of non-negative integers $\alpha,\beta$ with $\alpha+\beta\neq 0$ and $f_1\in\mathcal{P}$ such that
$$
\nu (f)=\nu (x^\alpha y^\beta f_1)= -(\alpha+\beta \delta) \nu(u)-\beta\nu(v) + \nu (f_1)\leq 0.
$$

If $\delta=0$ the proof is analogous, and the non-positivity of $\nu$ for the case when $p$ is a general point can be proved in a similar way after assuming that $p$ has coordinates $(0:1;0,1)$ and considering local coordinates $\{u,v\}=\big\{\frac{X_0}{X_1},\frac{Y_0}{X_1^\delta Y_1}\big\}$ in the affine open set $U_{11}$ and $\{x,y\}=\big\{\frac{X_1}{X_0},\frac{Y_0}{X_0^\delta Y_1}\big\}$ in $U_{01}$.

Now we are going to prove that Item $(a)$ implies Item $(b)$. Assume by contradiction that the divisor $\Lambda_n$ is not  nef and, therefore, that there exists an effective divisor $C$ such that $\Lambda_n\cdot C<0$. This implies that, with the above notation, if $p$ is a special point  --or $p\in\mathbb{F}_0$--, (respectively, $p$ is a general point), then there exists $f\in\mathcal{O}_{\mathbb{F}_\delta}(U_{11})$ (respectively, $f\in\mathcal{O}_{\mathbb{F}_\delta}(U_{01})$) such that $-\nu(f)=\Lambda_n\cdot C<0$, a contradiction.

The fact that Item (b) implies Item (c) follows easily from previous computations given in the proof of Lemma \ref{Lemma_2_casoespecial}.

Let us prove that Item (d) can be deduced from Item (c). Fix any ample divisor $H$ on the surface $Z$ and consider the set
$$
A(Z):=\{[D]\in \text{Pic}_\mathbb{Q}(Z) \ |\ [D]^2\geq 0 \mbox{ and }[H]\cdot [D] \geq 0\}.
$$
Recall that the above defined cone $S_1(Z)$ is generated by the classes $[\tilde{F}_1]$, $[\tilde{M}_0]$ and $[E_i]$, $1 \leq i \leq n$, and we are going to prove that
\begin{equation}
\label{Hh}
\overline{NE}(Z)=S_1(Z)+S_1^\vee(Z)=NE(Z)
\end{equation}
and
\begin{equation}
\label{Ii}
S_1^\vee(Z)\subseteq A(Z)\subseteq S_1(Z)
\end{equation}
hold, which shows Item (d). Our Hypothesis (c) means that $\Lambda_n^2 \geq 0$ and by Lemma \ref{Lemma_2_casoespecial}, one has that $\Lambda_i^2\geq 0,$ $1\leq i\leq n-1$. Proposition \ref{Prop_generatorsofdualcone_casespecial} proves the first inclusion in (\ref{Ii}) and the last one follows from the first one and the equality $A(Z)^\vee=A(Z)$ (that holds taking into account the Hodge index theorem \cite[Theorem 1.9]{Har}). It remains to prove the chain of equalities (\ref{Hh}). For a start, notice that $A(Z) \subseteq \overline{NE} (Z)$ by \cite[Lemma 1.20]{KolMor}.
Thus $S_1^\vee(Z) \subseteq \overline{NE}(Z)$. Now, if $[C]$ is the class of an irreducible curve on $Z$   and it is not one of the given generators of $S_1(Z)$, then $[C]\in S_1^\vee(Z)$ because otherwise $[C]\cdot [D]<0$ for $[D]\in S_1(Z)$ and $C$ and $D$ would have a common component. Therefore we have proved the chain (\ref{Hh}) with inclusions $\supseteq$ instead of equalities. Taking topological closures
we deduce that (\ref{Hh}) holds.

Finally Item (d) implies Item (b) by Proposition \ref{Prop_generatorsofdualcone_casespecial}, which concludes the proof.
\end{proof}

\begin{remark}\label{remmm}
The Hirzebruch surface $\mathbb{F}_1$ can be regarded as the projective space $\mathbb{P}^2$ with a point blown-up. If one considers the line at infinity $L$ in $\mathbb{P}^2$ and one regards $\mathbb{F}_1$ as the blow-up of $\mathbb{P}^2$ at a point of $L$, then it is not difficult to deduce that $M_0 = E_1$ and $F$ is a general divisor in $|\tilde{L}|$. As a consequence, Theorem \ref{Th1_caso_especial} allows us to provide equivalent conditions to the non-positivity of a valuation of $\mathbb{P}^2$ (see Section \ref{nonposneg}). In fact, our Theorem \ref{Th1_caso_especial} recovers Theorem 1 in \cite{GalMon} considering $\delta=1$ and $p_1$ a special point, and the above proof is an adaptation and extension to our more general situation of that of  \cite[Theorem 1]{GalMon}.
\end{remark}

\begin{remark}
\label{34}
Assume that $Z$ is a surface as in Theorem \ref{Th1_caso_especial} defined by a non-positive at infinity special valuation. Then, on the one hand, Lemma \ref{Lemma_2_casoespecial} and Theorem \ref{Th1_caso_especial} prove that the divisorial valuation $\nu_i$ defined by any exceptional divisor $E_i$ given by \eqref{Eq_sequencepointblowingups} is also special and non-positive at infinity. On the other hand, every divisor $\Lambda_i$, $1\leq i\leq n,$ is effective. Indeed, under these conditions, the expression of $\Lambda_i,$ in the basis of strict transforms $\{\tilde{F}_1,\tilde{M}_0\} \cup \{E_j\}_{j=1}^n$, is
$$
\Lambda_i=(\Lambda_i\cdot M^*)\tilde{F}_1+(\Lambda_i\cdot F^*)\tilde{M}_0 + \sum_{j=1}^n(\Lambda_i\cdot\Lambda_j)E_j,
$$
which is effective because $\Lambda_i$ is nef.
\end{remark}

We conclude this section by stating a characterization result for negative at infinity special valuations of Hirzebruch surfaces. Argumenting as in Remark \ref{remmm}, one can see that our result also proves \cite[Theorem 2]{GalMon}.

\begin{theorem}\label{Th2_caso_especial}
Let $\nu$ (respectively, $Z$, $\Lambda_n$) a divisorial valuation (respectively, a surface, a divisor on $Z$) as in Theorem \ref{Th1_caso_especial}. Then, the following conditions are equivalent:
\begin{itemize}
\item[(a)] The valuation $\nu$ is negative at infinity.
\item[(b)] It holds that either $2a_nb_n + b_n^2 \delta > [\text{\emph{vol}}(\nu)]^{-1}$,  or  $\;2a_nb_n +b_n^2 \delta = [\text{\emph{vol}}(\nu)]^{-1}$ and the Iitaka dimension of the divisor $\Lambda_n$ vanishes.
\item[(c)] The inequality $\Lambda_n\cdot \tilde{C}>0$ holds for the strict transform on $Z,$ $\tilde{C},$ of any  curve $C$ on $\mathbb{F}_\delta,$ $C\neq F_1,M_0$.
\end{itemize}
\end{theorem}

\begin{proof}
For a start, we recall that the Iitaka dimension \cite{Iit} of a divisor $D$ on $Z$ is the maximum of the projective dimensions of the closures of the images of the rational maps defined by the complete linear systems $|nD|$, when $n$ runs over those positive integers $m$ such that $H^0(Z,\mathcal{O}_Z(mD)) \neq 0$.

We assume that $p_1$ is a special point. The other cases can be proved similarly. Assume also, without loss of generality, that $p_1$ has coordinates $(1:0;1,0)$ and consider the same notations as in the proof of Theorem \ref{Th1_caso_especial}.

We start by proving by contradiction that Statement (b) can be deduced from Statement (a). Hence, assume that (a) holds but (b) is false (what means, taking into account Theorem \ref{Th1_caso_especial}, that $\Lambda_n^2=0$ and $\dim |m\Lambda_n|>0$ for $m$ large enough). Therefore, there exists $f \in \mathcal{P}$ such that the class of $m\Lambda_n-\tilde{C}_f$ is effective for $m$ large enough. This implies that
$$0\leq \Lambda_n\cdot (m\Lambda_n-\tilde{C}_f)= m\Lambda_n^2 -\Lambda_n{\cdot}\tilde{C}_f=-\Lambda_n\cdot\tilde{C}_f.$$
Hence $0=\Lambda_n\cdot\tilde{C}_f=-\nu(f)$ because $\Lambda_n$ is nef (by Theorem \ref{Th1_caso_especial}), and this fact contradicts (a).



To prove that Statement (b) implies Statement (c), we reason again by contradiction and consider $C$ an integral curve on $\mathbb{F}_\delta$ different from $F_1$ and $M_0$, and such that $\Lambda_n\cdot \tilde{C}\leq 0$. In fact $\Lambda_n\cdot \tilde{C}= 0$ because, by Theorem \ref{Th1_caso_especial}, $\Lambda_n$ is nef. Let $\mathcal{F}$ be the face of the cone of curves of $Z$ spanned by the classes $[\tilde{F}_1],[\tilde{M}_0],$ $[E_1],\ldots,[E_{n-1}],$ that is, $ \mathcal{F}=[\Lambda_n]^\perp\cap NE(Z).$ It is clear that $[\tilde{C}] \in \mathcal{F}$ and, since the extremal rays of $NE(Z)$ are generated by classes of irreducible curves with negative self-intersection, $\tilde{C}^2=0$. $\tilde{C}$ is nef, so $[\tilde{C}]^\perp\cap NE(Z)$ is a face of $NE(Z)$ which contains $[\tilde{C}]$ and, then, it must coincide with $[\Lambda_n]^\perp\cap NE(Z)$. Indeed, this is a consequence of the fact that, in suitable coordinates, $A(Z)$ is the projective cone over an Euclidean ball $B$ (by the Hodge index theorem \cite[Theorem 1.9]{Har}) and $B$ is strictly convex.
Then, $\tilde{C}$  is linearly equivalent to a multiple of $\Lambda_n$ and, by Remark \ref{34}, we get a contradiction.

To finish, the fact that Statement (c) implies Statement (a) can be proved as in Theorem \ref{Th1_caso_especial} when proving that Item (b) implies Item (a).

\end{proof}

\begin{remark}
The concepts of non-positivity (and negativity) at infinity of valuations of $\mathbb{P}^2$ and $\mathbb{F}_\delta$ are different. For instance, by \cite[Theorem 1]{GalMon} there is no non-positive at infinity divisorial valuation of $\mathbb{P}^2$ with maximal contact values $3$, $11$ and $122$; however Theorem \ref{Th1_caso_especial} proves the existence of non-positive at infinity valuations of $\mathbb{F}_\delta$ with those maximal contact values, when $\delta \geq 2$.
\end{remark}

\section{The sign at infinity of non-special valuations}
\label{gener}

This section gives results that characterize the non-positivity and negativity at infinity of non-special divisorial valuations (see Definition \ref{SG2}) of a Hirzebruch surface $\mathbb{F}_\delta$, $\delta >0$.

Notice that, when considering non-special valuations, there exists a unique irreducible section that is linearly equivalent to $M$ and whose strict transform on $Z$ has negative self-intersection. We will denote this section by $M_1$. Notice that its class gives an extremal ray of the cone $NE(Z)$.

For reaching our objectives, we need to describe the dual cone of the strongly convex cone of $\text{Pic}_\mathbb{Q}(Z)$, $S_2(Z)$, generated by the classes $[\tilde{F}_1],[\tilde{M}_0],[\tilde{M}_1]$ and $[E_i]$, $1 \leq i \leq n$. Our first result is a lemma which we will use in the forthcoming Proposition \ref{Prop_generatorsofdualcone_casegeneral} (that gives generators for the mentioned dual cone).

\begin{lemma}\label{Lemma_expressionM1_basistildeFM0}
The class of the strict transform of $M_1$ on $Z$, $[\tilde{M}_1]$, can be written as
\begin{align*}
[\tilde{M}_1]=\delta[\tilde{F}_1]+  [\tilde{M}_0]+ (\delta - 1)[E_1]+(\delta -2)[E_2]+\ldots\\ + [E_{\delta-1}]+ d_{\delta+1}[E_{\delta+1}]+\ldots + d_{n}[E_n],
\end{align*}
where $d_i \in \mathbb{Z}$ and $d_i\leq -1$ for all $i=\delta + 1,\delta + 2,\ldots, n$.
\end{lemma}
\begin{proof}
It is clear that we can write  $[\tilde{M}_1]$ as
\begin{equation*}
[\tilde{M}_1]=d_{01}[\tilde{F}_1]+ d_{02} [\tilde{M}_0]+ d_1[E_1]+\ldots + d_{n}[E_n],
\end{equation*}
for some values $d_{01},d_{02},d_i\in\mathbb{Z}$, $1 \leq i \leq n$. Now, using the equalities
\begin{equation*}
[\tilde{F}_1]=[F^*]-[E_1^*], \ [\tilde{M}_0]=\delta [F^*]+[M^*] \text{ and } \ [E_i]=[E_i^*]-\sum_{p_j\to p_i}[E_j^*],
\end{equation*}
we can compare the above expression with the equality $[\tilde{M}_1]=[M^*]- \sum_{j=1}^{i_{M_1}}[E_j^*],$ where $i_{M_1}$ is the index of the last point of the configuration of infinitely near points given by $\nu$, $\mathcal{C}_\nu$, through which $\tilde{M}_1$ goes. This gives rise to a system of linear equations in the variables $d_{01},d_{02},\{d_i\}_{i=1}^{n}$, whose first equations are
\begin{align*}
d_{01}-\delta d_{02}=0, \, d_{02}=1, \ d_1 - d_{01}=-1, \,  d_2 - d_1=-1,  \ldots, \, d_{\delta -1} - d_{\delta -2} = -1,\\ d_{\delta} - d_{\delta-1}=-1.
\end{align*}
These equations determine the values of $d_{01},d_{02},d_i$ for $i \in \{1,2, \ldots, \delta\}$, that coincide with those given in the statement. The fact that $d_i \leq -1$ for $i \geq \delta +1$ follows from considering the remaining equations and recalling that non-free points can only appear when $j > \delta + 1$.
\end{proof}

\begin{proposition}\label{Prop_generatorsofdualcone_casegeneral}
Let $Z$ be the surface given by a non-special divisorial valuation and $S_2(Z)$ the cone of $\text{Pic}_\mathbb{Q}(Z)$ defined before Lemma \ref{Lemma_expressionM1_basistildeFM0}. Then, the dual cone of $S_2(Z)$, $S_2^\vee(Z)$, is generated by the following of classes of divisors: $[F^*], [M^*]$, $\{[\Theta_i]\}_{i=1}^{\delta},$ $ \{ [\Delta_i] \}_{i=\delta+1}^{n}, $ $\{ [\Gamma_i] \}_{i=\delta+1}^n $ and $\{[\Upsilon_{ik}]\}_{i=\delta+1,k=1,\ldots,\delta -1}^n$, where
\begin{equation*}
\begin{array}{c}
\Theta_i:=  b_iM^* - \displaystyle\sum_{j=1}^i\text{\emph{mult}}_{p_j}(\varphi_i)E_j^*,\\
\Delta_i:=(-\delta b_i+c_i)F^* +b_iM^* - \displaystyle\sum_{j=1}^i \text{\emph{mult}}_{p_j}(\varphi_i)E_j^*,\\
\Gamma_i:= c_i M^* - \displaystyle\sum_{j=1}^i\left(\delta \, \text{\emph{mult}}_{p_j}(\varphi_i)\right)E_j^*,\text{ and }\\
\end{array}
\end{equation*}
\begin{equation*}
\begin{array}{c}
\Upsilon_{ik}:=(c_i-kb_i)M^* - \displaystyle\sum_{j=1}^{k}(c_i-kb_i)E_j^*-\displaystyle\sum_{j=k+1}^i\left((\delta-k) \text{\emph{mult}}_{p_j}(\varphi_i)\right)E_j^*,
\end{array}
\end{equation*}
and where
\[
b_i :=(\varphi_{F_1},\varphi_i)_p, 1 \leq i \leq n,\; \mbox{ and } \; c_i :=(\varphi_{M_1},\varphi_i)_p, \delta+1 \leq i \leq n.
\]

\end{proposition}

\begin{proof}
It suffices to consider every $(n+1)$-dimensional linear subspace $H$ generated by elements of $S_2(Z)$ and check whether $H^\perp$ is generated by an element of $S_2^{\vee}(Z)$. We will see that these generators will be those in the statement.

Denote by $\langle S\rangle$ the linear subspace generated by a set $S\subseteq \text{Pic}_\mathbb{Q}(Z)$. Then,
$$\langle \{[\tilde{F}_1]\}\cup \{[E_i]\}_{i=1}^n \rangle^\perp=\langle [F^*]\rangle,\;\; \langle \{[\tilde{M}_0]\}\cup \{[E_i]\}_{i=1}^n \rangle^\perp=\langle [M^*] \rangle,$$
and $[F^*],  [M^*]\in S_2^{\vee}(Z)$. Moreover $\langle \{[\tilde{M}_1]\}\cup \{[E_i]\}_{i=1}^n \rangle^\perp$ is not generated by an element in $S_2^{\vee}(Z)$.


We have studied spaces $H$ whose generators contain all the classes $[E_i]$. Now we will treat the cases where a class $[E_i]$, $1\leq i\leq n,$ is not considered. Let us start with the linear space $\langle \{[\tilde{F}_1],[\tilde{M}_0] \} \cup \{[E_j]\}_{1\leq j\leq n, j\neq i} \rangle$, set
$$[D_i]=d_{i01}[F^*]+d_{i02}[M^*]+d_{i1}[E_1^*]+\cdots+d_{in}[E_n^*]\in \text{Pic}_\mathbb{Q}(Z)$$
with arbitrary coefficients and impose the conditions:
$$
[D_i]\cdot [\tilde{F}_1]=0,[ D_i ]\cdot [ \tilde{M}_0 ]= 0,[D_i]\cdot [E_j]=0,[ D_i]\cdot [ \tilde{M}_1 ]\geq 0\text{ and }[D_i]\cdot [E_i]\geq 0.
$$
Then we obtain the system of equalities and inequalities:
\begin{equation*}
\begin{array}{c}
d_{i02} + d_{i1}=0, \, d_{i01}+\delta d_{i02} -\delta d_{i02}= 0, \, -d_{ij}+\displaystyle\sum_{p_s\rightarrow p_j}d_{is}=0,\\
d_{i01}+\delta d_{i02} + \displaystyle\sum_{j=1}^{\min\{i,i_{M_1}\}}d_{ij}\geq 0 \text{ and } -d_{ii}+\displaystyle\sum_{p_s\rightarrow p_i}d_{is}\geq 0,
\end{array}
\end{equation*}
where $i_{M_1}$ is the index defined in the proof of Lemma \ref{Lemma_expressionM1_basistildeFM0}. Solving it, we obtain $d_{ii}=-1;d_{ij}=\sum_{p_s\rightarrow p_j}d_{is},$ $1\leq j\leq i-1;$ $d_{ij}=0$, $i+1\leq j\leq n$; $d_{i01}=0;$ and $d_{i02}=-d_{i1}$. This proves that $d_{ij}=-\text{mult}_{p_j}(\varphi_i)$ holds and also
$$
\delta \ \mbox{mult}_{p_1}(\varphi_i) - \sum_{j=1}^{\min\{i,i_{M_1}\}} \mbox{mult}_{p_j}(\varphi_i) \geq 0
$$
by our first inequality, which shows that the classes $\{[\Theta_i]\}_{1\leq i\leq \delta}$ in the statement give generators of the dual cone $S_2^\vee(Z).$

Reasoning as above for the subspace $\langle \{ [\tilde{F}_1],[\tilde{M}_1]\}\cup \{ [E_j]\}_{1\leq j\leq n, j\neq i} \rangle$ and with the same notation, we get the system of equalities and inequalities:
\begin{equation*}
\begin{array}{c}
d_{i02}+d_{i1}=0, \, d_{i01} + \delta d_{i02} + \displaystyle\sum_{j=1}^{\min\{i,i_{M_1}\}} d_{ij}=0, \, -d_{ij}+\displaystyle\sum_{p_s\rightarrow p_j}d_{is}=0,\\
d_{i01}+\delta d_{i02} -\delta d_{i02}\geq 0 \text{ and } -d_{ii}+\displaystyle\sum_{p_s\rightarrow p_i}d_{is}\geq 0.
\end{array}
\end{equation*}
Here, the equality $d_{ij}=-\text{mult}_{p_j}(\varphi_i)$ is again true and the first inequality means that
$$
\sum_{j=1}^{\min\{i,i_{M_1}\}} \mbox{mult}_{p_j}(\varphi_i) - \delta \ \mbox{mult}_{p_1}(\varphi_i) \geq 0
$$
must hold. As a consequence, we have proved that the classes $\{[\Delta_i]\}_{\delta+1\leq i\leq n}$ in the statement give extremal rays of $S_2^\vee(Z).$

Repeating the procedure with $\langle \{ [\tilde{M}_0],[\tilde{M}_1]\}\cup \{ [E_j]\}_{1\leq j\leq n, j\neq i} \rangle$, the obtained system is
\begin{equation*}
\begin{array}{c}
d_{i01}+\delta d_{i02} -\delta d_{i02}= 0, \, d_{i01} + \delta d_{i02} + \displaystyle\sum_{j=1}^{\min\{i,i_{M_1}\}} d_{ij}=0, \, -d_{ij}+\displaystyle\sum_{p_s\rightarrow p_j}d_{is}=0,\\
d_{i02} + d_{i1}\geq 0 \text{ and } -d_{ii}+\displaystyle\sum_{p_s\rightarrow p_i}d_{is}\geq 0.
\end{array}
\end{equation*}
This proves, on the one hand, that $d_{ii}=-1;d_{ij}=\sum_{p_s\rightarrow p_j}d_{is},$  $1\leq j\leq i-1;$ $d_{ij}=0$, $i+1\leq j \leq n$; $d_{i01}=0;$ and $d_{i02}=(1/\delta)\sum_{j=1}^{\min\{i,i_{M_1}\}} -d_{ij} $. On the other hand, reasoning as above, $d_{ij}=\delta\;\text{mult}_{p_j}(\varphi_i)$ and then
$$
\sum_{j=1}^{\min\{i,i_{M_1}\}} \mbox{mult}_{p_j}(\varphi_i) - \delta \ \mbox{mult}_{p_1}(\varphi_i) \geq 0,
$$
which shows that the set of classes $\{[\Gamma_i]\}_{\delta+1\leq i\leq n}$ gives  generators of $S_2^\vee(Z)$.

It only remains to consider those subspaces $\langle \{[\tilde{F}_1], [\tilde{M}_0],[\tilde{M}_1]\}\cup \{ [E_j]\}_{j\in\{1,2,\ldots,n\}\setminus\{k,i\}} \rangle$ attached to pairs of indices $k, i$, $1\leq k<i\leq n$.
Lemma \ref{Lemma_expressionM1_basistildeFM0} proves the $(n+1)$-dimensionality of these subspaces. Our computations depend on two indices $i$ and $k$. So, we will write
$$[D_{ik}]=d_{ik01}[F^*]+ d_{ik02}[M^*]+d_{ik1}[E_1^*]+ d_{ik2}[E_2^*]+\cdots+ d_{ikn}[E_n^*].$$
We must impose the following conditions:
\begin{align*}
[D_{ik}]\cdot [\tilde{F}_1]=0, [D_{ik}]\cdot [\tilde{M}_0]=0,[D_{ik}]\cdot [\tilde{M}_1]=0,[D_{ik}]\cdot [E_j]=0,[D_{ik}]\cdot [E_k]\geq 0 \\
\text{ and }[D_{ik}]\cdot [E_i]\geq 0,
\end{align*}
which give the equivalent system
\begin{equation*}
\begin{array}{c}
d_{ik02}+d_{ik1}=0, \, d_{ik01}=0, \, d_{ik01} + \delta d_{ik02} + \displaystyle\sum_{j=1}^{\min\{i,i_{M_1}\}} d_{ikj}=0,\\
-d_{ikj}+\displaystyle\sum_{p_s\rightarrow p_j}d_{iks}=0,-d_{ikk}+\displaystyle\sum_{p_s\rightarrow p_k}d_{iks}\geq 0 \text{ and } -d_{iki}+\displaystyle\sum_{p_s\rightarrow p_i}d_{iks}\geq 0.
\end{array}
\end{equation*}
To solve it we can assume that the inequalities are strict because, otherwise, we would obtain that $[D_{ik}]$ either vanishes or it gives the class $[\Theta_\delta]$. Indeed, if both inequalities are equalities, then $[D_{ik}]=0.$ Otherwise, taking into account that the first $\delta +1$ points in $\mathcal{C}_\nu$ are free, by considering the third equality and $\delta + 1\leq i_{M_1}$, it holds that one of the indices $i$ or $k$ equals $\delta$. This shows that we obtain $[\Theta_\delta]$ as a solution.

The solutions of the system satisfy  that $d_{iki}=-1;-d_{ikk}> -\sum_{p_s\rightarrow p_k}d_{iks}$; $ d_{ikj}=\sum_{p_s\rightarrow p_j}d_{iks},$  $1\leq j\neq k\leq i-1;$ $d_{ikj}=0$,  $i+1\leq j \leq n$; $d_{ik01}=0;d_{ik02}=-d_{ik1};$ and it must hold that
\begin{equation}\label{Cond_prop_gene_casogeneralM1negativo_upsilon}
-\delta d_{ik1}=-\sum_{j=1}^{\min\{i,i_{M_1}\}} d_{ikj},
\end{equation}
by the third equation. Note that, for $k+1\leq j\leq i$, $d_{ikj}=-\text{mult}_{p_j}(\varphi_i)$ up to a positive factor, and also that $-d_{ikk}> -\sum_{p_s\rightarrow p_k}d_{iks}\geq 0$.

The indices $i$ and $k$ must satisfy that $1\leq k\leq \delta -1$ and $\delta+1\leq i\leq n$. Indeed, with respect to $k$ and  reasoning by contradiction, suppose that $k\geq \delta.$ By hypothesis, $k<i,$ $\delta +1\leq i_{M_1}$, and $d_{ikj}=d_{ik\delta}$ for $1 \leq j \leq \delta -1$, because the first $\delta + 1$ points in $\mathcal{C}_\nu$ are free, then
$$
-\sum_{j=1}^{\min\{i,i_{M_1}\}} d_{ikj}= -\delta d_{ik1} -\sum_{j=\delta +1}^{\min\{i,i_{M_1}\}} d_{ikj} ,
$$
where $-\sum_{j=\delta +1}^{\min\{i,i_{M_1}\}} d_{ikj}> 0$, which does not hold by Equality  \eqref{Cond_prop_gene_casogeneralM1negativo_upsilon}. Notice that this equality is true by our imposed equalities. With respect to the index $i$, again reasoning by contradiction, suppose that $i\leq \delta.$ As $1\leq k\leq \delta -1$, Equation \eqref{Cond_prop_gene_casogeneralM1negativo_upsilon} is equivalent to
\begin{equation}\label{Cond_prop_gene_casogeneralM1negativo_upsilon_consecuencia}
-(\delta -k) d_{ikk}=-\sum_{j=k+1}^{\min\{i,i_{M_1}\}} d_{ikj},
\end{equation}
because $d_{ikj}=d_{ikk}$ for $1\leq j \leq k$. This implies that $-(\delta -k) d_{ikk}=-(i-k)d_{ikk+1}$,  which is a contradiction since $-d_{ikk} > -d_{ikk+1}$.

Equality \eqref{Cond_prop_gene_casogeneralM1negativo_upsilon} also gives us the value of $d_{ikk}$, which can be obtained from the following chain of equalities:
$$
d_{ikk}=\delta d_{ik1}-\sum_{j=1,\; j\neq k}^{\min\{i,i_{M_1}\}}d_{ikj}= \delta d_{ikk} - (k-1)d_{ikk}-\sum_{j=k+1}^{\min\{i,i_{M_1}\}}d_{ikj}.
$$

Thus, if we take $d_{ikj}=-(\delta -k)\mbox{mult}_{p_j}(\varphi_i),$ $k+1\leq j \leq i,$ one gets that
$$
d_{ik1}=\ldots =d_{ikk}=\dfrac{-(\delta -k) \sum_{j=k+1}^{\min\{i,i_{M_1}\}} \mbox{mult}_{p_j}(\varphi_i)}{(\delta-k)}=-\sum_{j=k+1}^{\min\{i,i_{M_1}\}} \mbox{mult}_{p_j}(\varphi_i),
$$
and the coefficient of $[M^*]$ is $d_{ik02}=-d_{ik1}$.

As a result, we have that $[D_{ik}]=[\Upsilon_{ik}],$ where
\begin{align*}
\Upsilon_{ik}:=\left(\sum_{j=k+1}^{\min\{i,i_{M_1}\}} \mbox{mult}_{p_j}(\varphi_i)\right)M^* - \sum_{j=1}^{k}\left(\sum_{s=k+1}^{\min\{i,i_{M_1}\}} \mbox{mult}_{p_s}(\varphi_{i})\right)E_j^*\\
-\sum_{j=k+1}^i\left((\delta-k) \mbox{mult}_{p_j}(\varphi_i)\right)E_j^*,
\end{align*}
where $\delta + 1\leq i \leq n $ and $1\leq k\leq \delta -1$. This finishes the proof.
\end{proof}

\begin{remark}
From the above proof, it can be deduced that, when considering the surface $\mathbb{F}_1$ and a non-special valuation $\nu$,  no class $[\Upsilon_{ik}]$ appears as a generator of $S_2^\vee(Z)$.
\end{remark}

We are interested in determining conditions under which the generators of the cone $NE(Z)$ of the surfaces $Z$ given by non-special valuations are known. The divisors introduced in Proposition \ref{Prop_generatorsofdualcone_casegeneral} will be important for this purpose. Next lemma states some of their properties.

\begin{lemma}\label{Lemma_2_casogeneral}
Let $Z$ (respectively, $\nu$) be a rational surface (respectively, valuation) as in Proposition \ref{Prop_generatorsofdualcone_casegeneral}. Consider the set of divisors there defined. Then $\Delta_{\delta +1}^2 > 0$, $\Gamma_{\delta +1}^2>0$ and $\Upsilon_{\delta+1 k}^2>0$ for all $ k\in\{1,2,\ldots,\delta -1\}$. In addition, for any index $i\in\{\delta+2,\delta+3,\ldots,n\}$ such that $\Delta_i^2\geq 0$ (respectively, $\Gamma_i^2\geq 0$, $\Upsilon_{ik}^2\geq 0$), the following properties are satisfied:
\begin{itemize}
\item[(a)]If $p_i$ is a satellite point of the configuration $\mathcal{C}_\nu$ that $\nu$ defines, it holds $\Delta_i^2>0$ (respectively, $\Gamma_i^2>0$, $\Upsilon_{ik}^2>0$).
\item[(b)] $\Delta_{i-1}^2\geq 0$ (respectively, $\Gamma_{i-1}^2\geq 0$,$\Upsilon_{i-1 k}^2\geq 0$) and, moreover, if $\Delta_{i-1}^2=0$ (respectively, $\Gamma_{i-1}^2=0$,$\Upsilon_{i-1 k}^2=0$) then $p_i$ is a satellite point and $p_{i-1}$ is free.
\end{itemize}
\end{lemma}
\begin{proof}
To prove our first assertion, it suffices to notice that the following three equalities hold:
\begin{equation*}
\begin{array}{rl}
\Delta_{\delta+1}^2 \!\!\!\!&= 2 + \delta - (\delta + 1)=1>0,\\[2mm]
\Gamma_{\delta + 1}^2 \!\!\!&= \delta(\delta  + 1)^2 - \delta^2(\delta + 1)= \delta(\delta + 1)(\delta +1 - \delta)= \delta(\delta + 1)>0,\\[2mm]
\Upsilon_{\delta + 1 k}^2 \!\!\!\!&= \delta(\delta + 1-k)^2 - k(\delta +1 - k)^2 - (\delta + 1 -k)(\delta -k)^2\\[2mm]
 \!\!\!\!&= (\delta -k)(\delta +1 - k)^2 - (\delta -k)^2(\delta + 1-k) \\[2mm]
 \!\!\!\!&= (\delta-k)(\delta + 1 - k )[\delta-k+1-(\delta-k)]\\[2mm]
 \!\!\!\!&=(\delta - k)(\delta +1 - k)>0.
\end{array}
\end{equation*}

Items (a) and (b) can be proved reasoning as in the proof of Lemma \ref{Lemma_2_casoespecial}. Indeed, recalling that $g+2$ is the cardinality of the set of maximal contact values of $\nu$, the case $g=1$ follows as in that proof, and, when $g>1$, with notations as in that lemma and  in Proposition \ref{Prop_generatorsofdualcone_casegeneral}, it suffices to consider the following equalities and to reason again as we did in the mentioned Lemma \ref{Lemma_2_casoespecial}.
\begin{equation*}
\begin{array}{rl}
\Delta_n^2\!\!\!\!&= 2(-\delta b_n+c_n)b_n + b_n^2\delta - \overline{\beta}_{g+1}= 2b_nc_n - \delta b_n^2 - \overline{\beta}_{g+1},\\[2mm]
\!\!\!\!&=  e_{g-1}\left[\dfrac{2b_nc_n-\delta b_n^2}{e_{g-1}} -\overline{\beta}_g\right],\\[5mm]
\Gamma_n^2\!\!\!\!&=c_n^2\delta-\delta^2\overline{\beta}_{g+1},\;\;\; \mathrm{and}\\
\Upsilon_{nk}^2  \!\!\!\!& =(c_n-kb_n)^2\delta - k(c_n-kb_n)^2-(\delta-k)^2\displaystyle\sum_{j=k+1}^n\text{mult}_{p_j}^2(\varphi_n)) \\[5mm]
\!\!\!\!& = (\delta-k)[(c_n-kb_n)^2 -(\delta-k) (\overline{\beta}_{g+1} - kb_n^2)]\\[2mm]
\!\!\!\!& = (\delta-k)[c_n^2 - 2kc_nb_n + \delta kb_n^2 - (\delta -k )\overline{\beta}_{g+1}]\\[2mm]
\!\!\!\!& = (\delta-k)[c_n^2 - k( 2c_nb_n - \delta b_n^2) - (\delta - k)\overline{\beta}_{g+1}].
\end{array}
\end{equation*}
\end{proof}

\begin{remark}\label{Remark_inequalitiesofThetaDeltaGammaUpsilon}
Lemma \ref{Lemma_2_casogeneral} allows us to get numerical conditions which imply the non-nega\-ti\-vi\-ty of the self-intersection of the divisors, appearing in Proposition \ref{Prop_generatorsofdualcone_casegeneral}, whose classes generate the above defined dual cone $S_2^\vee(Z)$. Let us show which are those numerical conditions.

It is clear that the divisors $\Theta_i$, $1\leq i\leq \delta,$ satisfy $\Theta_i^2=\delta-i\geq 0$ because each $p_j,\ 1\leq j \leq i$, is a free point.

Now, $2b_nc_n-\delta b_n^2\geq [\text{vol}(\nu)]^{-1}$ implies that, for all $i\in\{1,2\ldots,n\}$, $\Delta_i^2\geq 0$, which is equivalent to the fact that $2b_ic_i-\delta b_i^2\geq [\text{vol}(\nu_i)]^{-1}$, where $\nu_i$ is the divisorial valuation defined in Remark \ref{34}. In a similar way, it holds that if $c_n^2\geq \delta [\text{vol}(\nu)]^{-1},$ then $\Gamma_i^2\geq 0$ or, equivalently, $c_i^2\geq \delta [\text{vol}(\nu_i)]^{-1}$ for all $i\in\{1,2,\ldots,n\}$.

Finally, for each integer $k$, $1\leq k\leq \delta -1,$ if one assumes $$c_n^2 - k(2c_nb_n-\delta b_n^2)\geq (\delta-k)[\text{vol}(\nu)]^{-1},$$ one can deduce that, for all $i\in\{1,2,\ldots, n\}$, $\Upsilon_{ik}^2\geq 0$ or, equivalently, $c_i^2-k(2c_ib_i-\delta b_i^2)\geq (\delta - k)[\text{vol}(\nu_i)]^{-1}$.
\end{remark}

Before giving our main result in this section, we need to state a last lemma.

\begin{lemma}\label{Lemma_onlyDelta2nonnegative}
Let $\nu$ be a non-special divisorial valuation of a Hirzebruch surface and $Z$ the surface that it defines. Consider the divisors $\Delta_i,\Gamma_i$ and $\Upsilon_{ik},$ $\delta+1\leq i\leq n;$ $1\leq k \leq \delta-1,$ given in Proposition \ref{Prop_generatorsofdualcone_casegeneral}. Then, for each index $i$, $\Delta_i^2\geq 0$ implies $\Gamma_i^2\geq 0$ and $\Upsilon_{ik}^2\geq 0$ for all $k\in\{1,2,\ldots,\delta-1\}$.
\end{lemma}
\begin{proof}
Our proof follows from the following two properties:

{\it Property} $1$: If the self-intersections of the divisors $\Delta_i$ and $\Upsilon_{i \delta -1}$ are non-negative, then the same property holds for the divisors $\Gamma_i$ and $\Upsilon_{ik},$ $1\leq k \leq \delta-1$.

{\it Property} $2$: If the self-intersection of the divisor $\Delta_i$ is non-negative, so is the self-intersection of $\Upsilon_{i \delta-1}$.

For proving Property $1$, our hypothesis are, by Remark \ref{Remark_inequalitiesofThetaDeltaGammaUpsilon},
\begin{equation} \label{primeraine}
[\text{vol}(\nu_i)]^{-1} \leq 2c_ib_i - \delta b_i^2 \;\;\text{ \ and \ }
\end{equation}
\begin{equation}\label{Cond_1_Lemma_onlyDelta2nonnegative}
(\delta - 1)( 2c_ib_i - \delta b_i^2 - [\text{vol}(\nu_i)]^{-1} )\leq c_i^2 - \delta [\text{vol}(\nu_i)]^{-1}.
\end{equation}
The inequality in \eqref{Cond_1_Lemma_onlyDelta2nonnegative} and the following one
\begin{equation*}
[\text{vol}(\nu_i)]^{-1} \leq c_i^2 - (\delta - 1)( 2c_ib_i - \delta b_i^2)
\end{equation*}
are equivalent. From this last inequality and the  one in \eqref{primeraine}, we get that $c_i^2\geq \delta[\text{vol}(\nu_i)]^{-1}$ and then $\Gamma_i^2\geq 0$. Finally, $\Upsilon_{ik}^2\geq 0$, $1\leq k \leq \delta -1,$ if and only if the inequality
$$
 k( 2c_ib_i - \delta b_i^2 - \overline{\beta}_{g+1}^{\ i}) \leq c_i^2 - \delta \overline{\beta}_{g+1}^{\ i}
$$
holds, fact that follows straightforwardly from the inequalities \eqref{primeraine} and \eqref{Cond_1_Lemma_onlyDelta2nonnegative}.

To conclude we prove Property $2$. It suffices to check that the following inequalities
\begin{equation}\label{Cond_2_Lemma_onlyDelta2nonnegative}
[\text{vol}(\nu_i)]^{-1}\leq 2c_ib_i - \delta b_i^2 < c_i^2 - (\delta - 1)( 2c_ib_i - \delta b_i^2)
\end{equation}
are true. In fact, the first inequality comes from our hypothesis $\Delta_i^2\geq 0$ and the inequality given by the first and the last sides in \eqref{Cond_2_Lemma_onlyDelta2nonnegative} allows us to show $\Upsilon_{i \delta-1}^2\geq 0$. To prove the second inequality in \eqref{Cond_2_Lemma_onlyDelta2nonnegative}, set $c_i=x$ and $b_i=b$ for simplicity. We are considering non-special valuations, which means that $x>\delta b$. In our new notation we want to prove that
$$
2bx -\delta b^2 < x^2- (\delta -1)(2bx-\delta b^2).
$$
This inequality is equivalent to
$$
0 < x^2 - (2b\delta)x + \delta^2 b^2,
$$
and it holds for all $x\neq \delta b$ since the point $(\delta b,0)$ is the vertex of the parabola given by the right-hand side of the inequality.
\end{proof}

Theorem \ref{Th1_caso_especial} considered special valuations of Hirzebruch surfaces. There we gave equivalent conditions to the non-positivity at infinity of valuations of that type. Our next result gives the corresponding conditions for non-special valuations. In fact, it gives an easy to check numerical and local condition, and two global properties concerning the surfaces that these valuations define. Before stating our result, we introduce the concepts of non-positive, and negative, at infinity, non-special valuation.

\begin{definition}
\label{SG2}
Let $\nu$ be a non-special divisorial valuation of a Hirzebruch surface $\mathbb{F}_\delta$ and keep the above notation. The valuation $\nu$ is called to be {\it non-positive} (respectively, {\it negative}) at infinity if $\nu(h) \leq 0$ (respectively, $\nu(h) <0$) for all $h\in\mathcal{O}_{\mathbb{F}_\delta}(\mathbb{F}_\delta\setminus(F_1\cup M_1))$ (respectively, $h\in\mathcal{O}_{\mathbb{F}_\delta}(\mathbb{F}_\delta\setminus(F_1\cup M_1))$ such that $h \notin k$).
\end{definition}

\begin{theorem}\label{Th1_caso_general}
Let $\nu$ be a non-special divisorial valuation of the fraction field of $R=\mathcal{O}_{\mathbb{F}_\delta,p}$ centered at $R$ and $\mathcal{C}_\nu=\{p_i\}_{i=1}^n$ the configuration of infinitely near points defined by $\nu$. Let $Z$ be the surface that $\nu$ defines and consider the divisor $\Delta_n$ on $Z$  defined in Proposition \ref{Prop_generatorsofdualcone_casegeneral}. Then, the following conditions are equivalent:
\begin{itemize}
\item[(a)] The valuation $\nu$ is non-positive at infinity.
\item[(b)] The divisor $\Delta_n$ is nef.
\item[(c)] It holds the following inequality $2c_nb_n-\delta b_n^2\geq [\text{\emph{vol}}(\nu)]^{-1}$.
\item[(d)] The cone of curves of $Z$ is generated by $[\tilde{F}_1],[\tilde{M}_0],[\tilde{M}_1],[E_1],[E_2],\ldots,[E_n]$.
\end{itemize}
\end{theorem}
\begin{proof}
Our proof uses a close reasoning to that of Theorem \ref{Th1_caso_especial}. Keeping the notation as in that theorem, we are going to give a sketch of the proof emphasizing only the main differences.

To prove that Item (a) can be deduced from Item (b), we can suppose that $p$ is a general point of $\mathbb{F}_\delta$ with coordinates $(0:1;0,1)$. Consider local coordinates $\{x,y\}=\big\{\frac{X_1}{X_0},\frac{X_0^\delta Y_1}{Y_0}\big\}$ in the affine open set $U_{00}$ and $\{u,v\}=\big\{\frac{X_0}{X_1},\frac{Y_0}{X_1^\delta Y_1}\big\}$ in $U_{11}$. Notice that, with our notation, $F_1$ and $M_1$ are defined by the equations $X_0=0$ and $Y_0=0$, $p\in U_{11},$ and $F_1$ and $M_1$ have local equations $u=0$ and $v=0$, respectively.

If now $\mathcal{S}$ denotes the set of non-constant polynomials in $\mathcal{O}_{\mathbb{F}_\delta}(U_{00})$ (up to multiplication by a nonzero element of $k$) such that neither $x$ nor $y$ divide them, $f\in\mathcal{S}$ satisfies
\begin{equation}\label{Ecu_th1_caso_generalM1negativo}
f(x,y)=f(1/u,u^\delta/ v)=\dfrac{h_f(u,v)}{u^{\deg_1(h_f)}v^{\deg_2(h_f)}},
\end{equation}
where $h_f(u,v)\in\mathcal{O}_{\mathbb{F}_\delta}(U_{11})$.

The bi-homogeneous polynomial $X_1^{\deg_1(h_f)+\delta \deg_2(h_f) }Y_1^{\deg_2(h_f)}h_f(\frac{X_0}{X_1},\frac{Y_0}{X_1^\delta Y_1})$ defines a curve $C_f$ on $\mathbb{F}_\delta$ of degree $(\deg_1(h_f),\deg_2(h_f))$ and $f\mapsto C_f$ is a one-to-one correspondence between $\mathcal{S}$ and the set of curves on $\mathbb{F}_\delta$ containing no curve $F_1,F',M_0,M_1$ as a component, where $F'$ and $M_0$ are defined by the equations $X_0=0$ and $Y_0=0$. Then $\Delta_n\cdot C_f = -\nu(f)$ and by Item (b),  $-\nu(f)\geq 0.$ The case when $f\in\mathcal{O}_{\mathbb{F}_\delta}(U_{00})$ and $x$ or $y$ or both  are factors of $f$ follows as in Theorem \ref{Th1_caso_especial} and Item (a) is proved.

A proof of the fact that Item (a) implies Item (b), Item (b) implies Item (c) and Item (d) implies Item (b) can be done as in Theorem \ref{Th1_caso_especial}.

To see that Item (c) implies Item (d), it suffices to notice that, by Lemmas \ref{Lemma_2_casogeneral} and \ref{Lemma_onlyDelta2nonnegative},
$$
S_2^\vee(Z)\subseteq \{[D]\in \text{Pic}_\mathbb{Q}(Z) \ |\ [D]^2\geq 0 \mbox{ and }[H]\cdot [D] \geq 0\}=:A(Z),
$$
where $S_2^\vee(Z)$ is the dual cone defined in Proposition \ref{Prop_generatorsofdualcone_casegeneral} and $H$ an ample divisor on $Z$. Finally, the fact
$$
S_2^\vee(Z)\subseteq A(Z) \subseteq (S_2^\vee(Z))^\vee = S_2(Z)
$$
and a reasoning as in Theorem \ref{Th1_caso_especial} completes our proof.
\end{proof}

An immediate consequence of the above result is the following one.

\begin{corollary}
Let $\nu$ be a non-positive at infinity non-special divisorial valuation of $\mathbb{F}_\delta$. Consider the divisorial valuations $\nu_i$ defined by the divisors $E_i$ associated to the simple sequence of point blowing-ups that $\nu$ defines. Then, the valuations $\nu_i$, $\delta+1\leq i\leq n-1$, are non-positive at infinity.
\end{corollary}

\begin{remark}
Let $Z$ be a surface as in Theorem \ref{Th1_caso_general} defined by a non-positive at infinity non-special valuation. Then, all the divisors $\Theta_i$, $i=1,2,\ldots,\delta$; $\Delta_i,\Gamma_i$ and $\Upsilon_{ik},$ $i=\delta+1,\delta+2,\ldots,n$ and $k=1,2,\ldots,\delta-1$ are effective. Indeed, under this assumption, all the divisors $\Theta_i,\Delta_i,\Gamma_i$ and $\Upsilon_{ik}$ can be expressed as
\begin{equation*}
\begin{array}{c}
\Theta_i=(\Theta_i\cdot F^*)\tilde{M}_1 + \displaystyle\sum_{j=1}^n(\Theta_i\cdot\Delta_j)E_j,\\[2mm]
\Delta_i=(\Delta_i\cdot M^*_0)\tilde{F}_1+(\Delta_i\cdot F^*)\tilde{M}_1 + \displaystyle\sum_{j=1}^n(\Delta_i\cdot\Delta_j)E_j,\\[2mm]
\Gamma_i=(\Gamma_i\cdot F^*)\tilde{M}_1 + \displaystyle\sum_{j=1}^n(\Gamma_i\cdot\Delta_j)E_j,\\[2mm]
\Upsilon_{ik}=(\Upsilon_{ik}\cdot F^*)\tilde{M}_1 + \displaystyle\sum_{j=1}^n(\Upsilon_{ik}\cdot\Delta_j)E_j,
\end{array}
\end{equation*}
which are effective divisors since $\Theta_i,\Delta_i,\Gamma_i$ and $\Upsilon_{ik}$ are nef divisors.
\end{remark}

\begin{example}
Let $\nu$ be a non-special divisorial valuation of the Hirzebruch surface $\mathbb{F}_2$ whose sequence of maximal contact values is $\{15,51,262,786\}$. Set $\mathcal{C}_\nu =\{p_i\}_{i=1}^{12}$ the configuration of infinitely near points of $\nu$. Let $F_1$ be the fiber of $\mathbb{F}_2$ that goes through $p_1$, $M_0$ the special section and $M_1$ the section that is linearly equivalent to $M$ and passes through $p_1,p_2$ and $p_3$. Then $b_{12}=15, c_{12}=45$ and $[\mathrm{vol}(\nu)]^{-1}=786$, and so, Item (c) in Theorem \ref{Th1_caso_general} is satisfied. Therefore, the cone of curves of the surface $Z$ defined by $\nu$ is generated by $\{[\tilde{F}_1], [\tilde{M_0}], [\tilde{M_1}]\} \cup \{E_i\}_{i=1}^{12}$ and the divisors $\Delta_i$, $1 \leq i \leq 12$, defined in Proposition \ref{Prop_generatorsofdualcone_casegeneral} are nef.
\end{example}

We finish this paper with a result that gives two equivalent properties to the fact of that a non-special valuation is negative at infinite. It can be proved as we did in Theorem \ref{Th2_caso_especial}.

\begin{theorem}
\label{Th2_caso_general}
Keep the same assumptions and notations as in Theorem \ref{Th1_caso_general}. Then the following conditions are equivalent:
\begin{itemize}
\item[(a)] The valuation $\nu$ is negative at infinity.
\item[(b)] It holds that either $2c_nb_n - b_n^2 \delta > [\text{\emph{vol}}(\nu)]^{-1}$, or  $2c_nb_n - b_n^2 \delta = [\text{\emph{vol}}(\nu)]^{-1}$ and the Iitaka dimension of the divisor $\Delta_n$ vanishes.
\item[(c)] The inequality $\Delta_n\cdot\tilde{C}>0$ is satisfied for the strict transform on $Z$, $\tilde{C}$, of any curve $C$ on $\mathbb{F}_\delta$, $C\neq F_1,M_1$.
\end{itemize}
\end{theorem}

\section*{Acknowledgements}
The authors thank M. Jonsson and W. Veys for valuable comments which help to improve the paper.


\bibliographystyle{plain}
\bibliography{BIBLIO}

\end{document}